\documentclass[12pt]{article}
\usepackage{type1cm}
\usepackage[T1]{fontenc}

\usepackage{amssymb,amsthm,amsmath}

\usepackage{latexsym}
\usepackage{makeidx}
\usepackage{showidx}
\usepackage[matrix,arrow,cmtip]{xypic}
\usepackage{colortbl}
\usepackage{fancybox}
\usepackage{fancyhdr}
\usepackage{wrapfig}

\newfont{\bb}{msbm10 at 12pt}
\newfont{\bbp}{msbm10 at 9pt}

\newtheorem{teorema}{Theorem}

\newtheorem{lema}{Lemma}
\newtheorem{corolario}{Corollary}
\newtheorem{definicion}{Definition}
\newtheorem{observacion}{Remark}

\def\r{\hbox{\bb R}}
\def\h{\hbox{\bb H}}

\def\fl{\longrightarrow}

\def\s{\hbox{\bb S}}

\newcommand{\campo}{\mathfrak{X}}

\newcommand{\set}[1]{\left\{#1\right\}}

\newcommand{\met}[2]{\langle #1,#2 \rangle }

\newcommand{\zb}{\bar{z}}

\newcommand{\Si}{\Sigma}

\begin{document}

\begin{center}
\rule{15cm}{1.5pt}\vspace{1cm}

{\LARGE \bf The Codazzi Equation for Surfaces}\\
\vspace{.5cm}

{\bf Juan A. Aledo$^a$, José M. Espinar${}^{b}$ and José A.
Gálvez${}^{c}$}\vspace{.5cm}

\rule{15cm}{1.5pt}
\end{center}

\noindent{\small ${}^a$Departamento de Matemáticas, Universidad de Castilla-La
Mancha, EPSA, 02071 Albacete, Spain; e-mail: JuanAngel.Aledo@uclm.es \\
${}^b$Institut de Math\'ematiques, Universit\'e
Paris VII, 175 Rue du Chevaleret, 75013 Paris, France;
e-mail: jespinar@ugr.es\\
${}^c$Departamento de Geometr\'\i
a y Topolog\'\i a, Facultad de Ciencias, Universidad de Granada, 18071 Granada, Spain;
e-mail: jagalvez@ugr.es}


\begin{abstract}
In this paper we develop an abstract theory for the Codazzi equation on surfaces, and use it as an analytic tool to derive new global results for surfaces in the space forms $\r^3$, $\s^3$ and $\h^3$. We give essentially sharp generalizations of some classical theorems of surface theory that mainly depend on the Codazzi equation, and we apply them to the study of Weingarten surfaces in space forms. In particular, we study existence of holomorphic quadratic differentials, uniqueness of immersed spheres in geometric problems, height estimates, and the geometry and uniqueness of complete or properly embedded Weingarten surfaces.
\end{abstract}

\section{Introduction}
The \emph{Codazzi equation} for an immersed surface $\Si$ in $\r^3$ yields
\begin{equation}\label{ecod}
\nabla_XSY-\nabla_YSX-S[X,Y]=0,\qquad X,Y\in\campo(\Si).
\end{equation}
Here $\nabla$ is the Levi-Civita connection of the first fundamental form $I$ of $\Si$ and $S$ is the shape operator, defined by $II(X,Y)=I(S(X),Y)$, where $II$ is the second fundamental form of the surface. This Codazzi equation is, together with the Gauss equation, one of the two classical integrability conditions for surfaces in $\r^3$, and it remains invariant if we substitute the ambient space $\r^3$ by other space form $\s^3$ or $\h^3$.

It is remarkable that some crucial results of surface theory in $\r^3$ only depend, in essence, of the Codazzi equation. This is the case, for instance, of Hopf's theorem (resp. Liebmann's theorem) on the uniqueness of round spheres among immersed constant mean curvature spheres (resp. among complete surfaces of constant positive curvature). This suggests the possibility of adapting these results to an abstract setting of \emph{Codazzi pairs} (i.e. pairs of real quadratic forms $(I,II)$ on a surface verifying \eqref{ecod}), and to explore its possible consequences in surface theory. The basic idea in this sense is to use the Codazzi pair $(I,II)$ as a geometric object in a \emph{non-standard} way, i.e. so that $(I,II)$ are no longer the first and second fundamental forms of a surface in a space form.

Motivated by this, our objective here is to develop an abstract theory for the Codazzi equation on surfaces, and use it subsequently as an analytic tool to derive new global results for surfaces in the space forms $\r^3$, $\s^3$ and $\h^3$.

Our results on Codazzi pairs here provide an extremely general extension of some classical theorems of surface theory that mainly depend on the Codazzi equation. But, moreover, this abstract approach has some very definite applications to the study of complete or properly embedded Weingarten surfaces in $\r^3$ or $\h^3$:

\begin{enumerate}
\item
It reveals the existence of holomorphic quadratic differentials for some classes of surfaces in space forms (and also in product spaces $\s^2\times \r$, $\h^2\times \r$, see \cite{AEG1}).
 \item
It unifies the proof of apparently non related theorems. For example, it shows that uniqueness in the Christoffel problem in $\r^3$ is basically equivalent to the Bonnet theorem on uniqueness of immersed spheres with prescribed mean curvature.
 \item
It gives an analytic tool to prove uniqueness results for complete or compact Weingarten surfaces in space forms.
\end{enumerate}
These applications show the flexibility of the use of Codazzi pairs in surface theory, and suggest the possibility of obtaining further global results with the techniques employed here.

We have organized this paper as follows. We shall start by reminding in Section \ref{s2} the definitions of fundamental pair, Codazzi pair, and some of their associated invariants such as the mean curvature $H$, the extrinsic curvature $K$ and the Hopf differential. We prove in Theorem \ref{elteorema} that a topological sphere $\Sigma$ endowed with a Codazzi pair satisfying a general Weingarten relationship $W(H,K)=0$ must be totally umbilical, if some necessary conditions are fulfilled by the functional $W$. This generalizes the previous Hopf theorem and Liebmann theorem.

This abstract treatment lets us apply
Theorem \ref{elteorema} to some seemingly unrelated situations. More specifically, as a consequence
of that result, we obtain generalizations of the Bonnet Theorem and
the theorem of uniqueness in the Christoffel problem. We point out that our proof
to the Christoffel problem is different from the classical approach, which uses integration theory on surfaces (see \cite{H,Sp}).

We will finish Section \ref{s2} proving that two Codazzi pairs $(I_i,II_i)$, $i = 1,2$, on a topological sphere $\Sigma$, such that $II_1=II_2$ and with the same positive extrinsic curvature must be isometric, that is, $I_1=I_2$. This result is a wide generalization of a classical result by Grove \cite{Gr} about rigidity of ovaloids
in $\r^3$.

In Section \ref{s3} we study when a real quadratic form $II$ on a Riemannian surface is conformal to the metric, even if the Codazzi equation is not satisfied. For that, we will define the Codazzi function on a surface associated to its induced metric $I$ and a real quadratic form $II$. This function will measure how far is the pair $(I,II)$ from satisfying the Codazzi equation.

We devote Section \ref{s4} to the fundamental relation between the Codazzi equation and the existence of holomorphic quadratic differentials. Thus, given a Codazzi pair on a surface $\Sigma$, we find, under certain conditions, the existence of a new Codazzi pair on $\Sigma$ whose Hopf differential is holomorphic. This new pair will provide geometric information about the initial one.

We particularize this result to the study of Codazzi pairs of special Weingarten type, that is, pairs satisfying $H=f(H^2-K)$ for a certain function $f$. The corresponding problem for surfaces in $\r^3$ and $\h^3$ was studied by Bryant in \cite{Br}. We will also prove that every Codazzi pair on a surface $\Sigma$ satisfying $H=f(H^2-K)$ can be recovered in terms of a metric on $\Sigma$ and a holomorphic quadratic form.

Finally, in Section \ref{s5}, we give some applications of our abstract approach to surfaces in space forms.
We begin by obtaining height estimates for a wide family of surfaces of elliptic type.
Although these estimates are not optimal, the existence of such height estimates with respect to planes constitute a fundamental tool for studying the behaviour of complete embedded surfaces.

Following the ideas developed by Rosenberg and Sa Earp in \cite{RoS}, we show that the theory developed by Korevaar, Kusner, Meeks and Solomon \cite{KKMS, KKS,M} for constant mean curvature surfaces in $\r^3$ and $\h^3$ remains valid for some families of surfaces satisfying the maximum principle (Theorem \ref{finales}).

In particular, when Theorem \ref{finales} is applied to a properly embedded Weingarten surface $\Sigma$ of elliptic type satisfying $H=f(H^2-K)$ in $\r^3$ or $\h^3$, we obtain: If $\Sigma$ has finite topology and $k$ ends, then $k\geq 2$, $\Sigma$ is rotational if $k=2$, and $\Sigma$ is contained in a slab if $k=3$.

To finish the paper, we study the problem of classifying Weingarten surfaces of elliptic type satisfying $H=f(H^2-K)$ in $\r^3$ such that $K$ does not change signs \cite{ST3}. We show that, in the above conditions, if $\Sigma$ is a complete surface with $K\geq0$ then it must be a totally umbilical sphere, a plane or a right circular cylinder, and if $\Sigma$ is properly embedded and $K\leq 0$, then it is a right circular cylinder or a surface of minimal type (i.e. $f(0)=0$).

\section{Fundamental pairs and Codazzi pairs}\label{s2}

Let us start this section by recalling some classical results about
fundamental pairs. A classical reference about this topic is
\cite{Mi}. Besides we point out that, although throughout this paper
we will assume that the differentiability used is always
$\cal{C}^{\infty}$, the differentiability requirements are much
lower.

We will denote by $\Si$ an orientable (and oriented) differentiable
surface. Otherwise we would work with its oriented two-sheeted
covering.

\begin{definicion}
A fundamental pair on $\Si$ is a pair of real quadratic forms
$(I,II)$ on $\Si$, where $I$ is a Riemannian metric.
\end{definicion}

Associated to a fundamental pair $(I,II)$ we define the shape
operator $S$ of the pair as
\begin{equation}\label{ii}
 II(X,Y)=I(S(X),Y)
\end{equation}
for any vector fields $X,Y$ on $\Si$.

Conversely, from (\ref{ii}) it becomes clear that the quadratic form
$II$ is totally determined by $I$ and $S$. In other words, giving a
fundamental pair on $\Si$ is equivalent to giving a Riemannian
metric on $\Si$ and a self-adjoint endomorphism $S$.

We define the {\it mean curvature}, the {\it extrinsic curvature}
and the {\it principal curvatures} of $(I,II)$ as half the trace,
the determinant and the eigenvalues of the endomorphism $S$,
respectively.

In particular, given local parameters $(x,y)$ on $\Si$ such that
$$
I=E\,dx^2+2F\,dxdy+G\,dy^2,\qquad II=e\,dx^2+2f\,dxdy+g\,dy^2,
$$
the mean curvature and the extrinsic curvature of the pair are
given, respectively, by
$$
H=H(I,II)=\frac{E g+G e-2F f}{2(EG-F^2)},\qquad K=K(I,II)=\frac{eg-f^2}{EG-F^2}.
$$
Moreover, the principal curvatures of the pair are
$H\pm\sqrt{H^2-K}$.

We will say that the pair $(I,II)$ is {\it umbilical} at $p\in \Si$
if $II$ is proportional to $I$ at $p$, or equivalently:
\begin{itemize}
\item if both principal curvatures coincide at $p$, or
\item if $S$ is proportional to the identity map on the tangent plane at $p$, or
\item if $H^2-K=0$ at $p$.
\end{itemize}

We define the {\it Hopf differential} of the fundamental pair
$(I,II)$ as the (2,0) part of $II$ for the Riemannian metric $I$. In
other words, if we consider $\Sigma$ as a Riemann surface with
respect to the metric $I$ and take $z$ a local conformal parameter,
then
\begin{equation}\label{parholomorfo}
\begin{array}
{c} I=2\lambda\,|dz|^2\\[2mm]
II=Q\,dz^2+2\lambda\,H\,|dz|^2+\overline{Q}\,d\zb^2.
\end{array}
\end{equation}
The quadratic form $Q\,dz^2$, which does not depend on the chosen
parameter, is known as the Hopf differential of the pair $(I,II)$.
We note that $(I,II)$ is umbilical at $p\in\Sigma$ if, and only if,
$Q(p)=0$.

All the above definitions can be understood as natural extensions of
the corresponding ones for isometric immersions of a Riemann surface
in a 3-dimensional ambient space, where $I$ plays the role of the
induced metric and $II$ the role of its second fundamental form.

A specially interesting case happens when the fundamental pair
satisfies, in an abstract way, the Codazzi equation for surfaces in
$\r^3$,

\begin{definicion}
We say that a fundamental pair $(I,II)$, with shape operator $S$, is a
Codazzi pair if
\begin{equation}\label{ecuacioncodazzi}
\nabla_XSY-\nabla_YSX-S[X,Y]=0,\qquad X,Y\in\campo(\Si),
\end{equation}
where $\nabla$ stands for the Levi-Civita connection associated to
the Riemannian metric $I$ and $\campo (\Si)$ is the set of
differentiable vector fields on $\Si$.
\end{definicion}

Many Codazzi pairs appear in a natural way in the study of surfaces.
For instance, the first and second fundamental forms of a surface
isometrically immersed in a 3-dimensional space form is a Codazzi
pair. The same occurs for spacelike surfaces in a 3-dimensional
Lorentzian space form. More generally, if the surface is immersed
in an $n$-dimensional (semi-Riemannian) space form and has a
parallel unit normal vector field $N$, then its induced metric and
its second fundamental form associated to $N$ make up a Codazzi
pair.

Classically, Codazzi pairs also arise in the study of harmonic maps.
Many others examples of Codazzi pairs also appear in
\cite{AEG1,Bi,Mi,Ol}. All of this shows that the results that we
present in this work can be used in many different contexts.

Many classical results in surface theory depend on the Codazzi
equation of the immersion. This fact allows to generalize such
results to the Codazzi pairs theory. Some examples of that, are
Hopf's results proving that the only surfaces immersed in $\r^3$
with constant mean curvature are totally umbilical. Analogously,
Liebmann proved that the only complete surfaces with positive constant
Gaussian curvature in $\r^3$ are totally umbilical spheres. Now, we obtain a generalization of both results to the wider family of Weingarten pairs,
which we define next
\begin{definicion}
We say that a fundamental pair $(I,II)$ on a surface $\Si$ is a
Weingarten pair if its mean and extrinsic curvatures, $H$ and $K$
respectively, satisfy a non trivial relationship
$$W(H,K)=0,$$
where $W$ is a differentiable function defined on an open set of
$\r^2$ containing the set of points $\{(H(p),K(p)):\ p\in\Si\}$.
\end{definicion}

\begin{teorema}\label{elteorema}
Let $(I,II)$ be a Codazzi pair on a surface $\Si$. If $(I,II)$ is a
Weingarten pair for a functional $W(x,y)$ such that
\begin{equation}
\label{hipotesis} W_x(t,t^2)+2t\,W_y(t,t^2)\neq0\qquad\mbox{for all
} t,
\end{equation}
then either the umbilical points of $(I,II)$ are isolated and of
negative index, or the pair is totally umbilical.

In particular, if $\Si$ is a topological sphere then $(I,II)$ is
totally umbilical.
\end{teorema}
Several proofs of this result when the ambient space is $\r^3$ or
$\h^3$ have been given by Hopf \cite{Ho}, Chern \cite{Ch}, Hartman
and Wintner \cite{HW}, Bryant \cite{Br} or Alencar, do Carmo and
Tribuzy \cite{AdCT}.
\begin{proof}
Let us consider $\Si$ as a Riemann surface with the conformal
structure induced by $I$. Given a local conformal parameter $z$, we
can write the fundamental pair $(I,II)$ as in (\ref{parholomorfo}).
Hence we have
\begin{equation}\label{levicivita}
\nabla_{\frac{\partial\
}{\partial z}}\frac{\partial\ }{\partial z}=\frac{\lambda_z}{\lambda}\
\frac{\partial\ }{\partial z},\qquad \nabla_{\frac{\partial\ }{\partial
z}}\frac{\partial\ }{\partial \zb}=0
\end{equation}
and the shape operator $S$ becomes
\begin{equation}\label{endomorfismoweingarten}
S\frac{\partial\ }{\partial z}=H\,\frac{\partial\ }{\partial z}+\frac{Q}{\lambda}\
\frac{\partial\ }{\partial \zb}.
\end{equation}
Consequently, if we take $X=\frac{\partial\ }{\partial z}$ and
$Y=\frac{\partial\ }{\partial \zb}$ in the Codazzi equation
(\ref{ecuacioncodazzi}) we get
\begin{equation}
\label{qzbarra}
Q_{\zb}=\lambda\,H_z.
\end{equation}
In addition, from (\ref{endomorfismoweingarten}) we obtain that the
extrinsic curvature is given by
\begin{equation}\label{modulodeq}
K=H^2-\frac{|Q|^2}{\lambda^2}.
\end{equation}
Thus, differentiating the equality $W(H,K)=0$ with respect to $z$
\begin{eqnarray*}
0&=&H_{z}\,W_x(H,K)+K_{z}\,W_y(H,K)\\
&=&H_{z}\,W_x(H,K)+\left(2H\,H_{z}-|Q|^2\left(\frac{1}{\lambda^2}\right)_{z}-
\frac{Q_{z}\overline{Q}+Q\,\overline{Q}_{z}}{\lambda^2}\right)\,W_y(H,K),
\end{eqnarray*}
and using (\ref{qzbarra})
$$
(W_x(H,K)+2H\,W_y(H,K))\,Q_{\zb}=\lambda\,W_y(H,K)\,\left(|Q|^2\left(\frac{1}{\lambda^2}\right)_{z}+
\frac{Q_{z}\overline{Q}+Q\,\overline{Q}_{z}}{\lambda^2}\right).
$$

Therefore, from (\ref{hipotesis}), if $p\in\Si$ is an umbilical
point (i.e. $Q(p)=0$ or equivalently $H^2=K$), there exists a
continuous function $h$ in a neighborhood $U$ of $p$ such that
$|Q_{\zb}|\leq h\,|Q|$ on $U$.

Hence, from \cite[Main Lemma]{AdCT} (see also \cite[Lemma 2.7.1]{Jost}), either $Q$ vanishes identically
on $U$ or $p$ is an isolated zero of negative index of $Q$.

In particular, if $\Si$ is a topological sphere, from the
Poincaré index Theorem we get that the Hopf differential
$Q\,dz^2$ must vanish identically on $\Si$, as we wanted to prove.
\end{proof}

\begin{observacion}
The above result can be globally used not only for topological
spheres. Indeed, if $\Si$ is a topological torus under the
assumptions of Theorem \ref{elteorema}, then we deduce that the pair
$(I,II)$ is either totally umbilical or umbilically free.
Analogously, if $\Si$ is a closed topological disk and its boundary
$\partial\Si$ is a line of curvature for $(I,II)$, then the pair is
totally umbilical.
\end{observacion}

It is well-known that the hypothesis (\ref{hipotesis}) cannot be
removed. Examples of this are the non totally umbilical rotational
spheres in any space form, since every rotational sphere is a
Weingarten surface.

The abstract use of Codazzi pairs allows us to see some classical
results, apparently non related, as immediate
consequences of Theorem \ref{elteorema}. Two good examples are
Bonnet Theorem and the uniqueness to the Christoffel problem in
$\r^3$, as we show next

\begin{corolario}\label{bonnet}
{\bf (Abstract Bonnet Theorem)} Let $\Si$ be a topological sphere
and $(I,II_1)$, $(I,II_2)$ two Codazzi pairs with the same Riemanian
metric $I$. If both pairs have the same mean curvature, then
$II_1=II_2$.
\end{corolario}

In $\r^3$, this result says that two isometric immersions from a
Riemannian sphere in $\r^3$ with the same mean curvature must
coincide, up to an isometry of the ambient space.

\begin{proof}
Since $(I,II_1)$ and $(I,II_2)$ are Codazzi pairs, so is the new
pair $(I,II_1-II_2)$. Besides, as $H(I,II_1)=H(I,II_2)$ the mean
curvature of $(I,II_1-II_2)$ vanishes identically. In particular, by
taking a local conformal parameter $z$, we can put (see
(\ref{parholomorfo}))
$$
I=2\lambda\,|dz|^2,\qquad II_1-II_2=Q\,dz^2+\overline{Q}\,d\zb^2.
$$
Thus, using Theorem \ref{elteorema} for the pair $(I,II_1-II_2)$ and
the functional $W(H,K)=H=0$, we get that $Q\equiv 0$, which finishes the proof.
\end{proof}

For a fundamental pair $(I,II)$ with mean and extrinsic curvatures
$H$ and $K$ respectively, the \emph{third fundamental form} is given
by $III=-K\,I+2H\,II$ (see, for instance, \cite{Mi}). In particular,
given a surface isometrically immersed in a 3-dimensional manifold
with first and second fundamental forms $I$ and $II$ respectively,
$III$ is nothing but its classical third fundamental form. In other
words, $III=\met{dN}{dN}$ where $N$ is a unit normal vector field on
the surface and $\met{}{}$ is the metric of the ambient space.

\begin{corolario} \label{chris} Let $\Si$ be a topological sphere and $(I_i,II_i)$, $i=1,2$, two Codazzi pairs
with mean curvature $H_i$ and extrinsic curvature $K_i$. If both
pairs have the same third fundamental form with $K_i(p)\neq 0$ for
all $p\in\Si$ and $\frac{H_1}{K_1}=\frac{H_2}{K_2}$, then
$(I_1,II_1)=(I_2,II_2)$.
\end{corolario}

When we particularize this result to the case of two isometric
immersions from a Riemannian sphere in $\r^3$ satisfying the
assumptions above, we get an easy proof of the uniqueness to the
Christoffel problem. Besides this proof is original in the sense
that the classical approaches to this problem use integration theory
on surfaces (see \cite{H,Sp}).

\begin{proof}
It is known \cite{Mi} that if $(I_i,II_i)$ is a Codazzi pair
with non vanishing extrinsic curvature, then $(III_i,II_i)$ is also
a Codazzi pair with mean curvature $\frac{H_i}{K_i}$. Consequently,
from Corollary \ref{bonnet} we deduce that $II_1=II_2$.

Thus, since $K(III_i,II_i)=\frac{1}{K_i}$ we have that $K_1=K_2$.
Finally, using that
$I_i=-\frac{1}{K_i}\,III_i+2\frac{H_i}{K_i}\,II_i$, it follows that
$I_1=I_2$.
\end{proof}

We observe that the previous proof is based in the simple fact that $(III,II)$ is a Codazzi pair. That is, the Bonnet theorem in $\r^3$ and the theorem of uniqueness of the Christoffel problem are a direct consequence of the Abstract Bonnet Theorem, when it is applied to the Codazzi pair $(I,II)$ or the Codazzi pair $(III,II)$, respectively.

In \cite{Gr} Grove proved that two ovaloids in $\r^3$ with the same
second fundamental form and extrinsic curvature are congruent. We give a different proof, generalizing that result to Codazzi pairs. The original proof by Grove involves techniques from integration theory on
surfaces.

\begin{teorema}\label{gengr}
Let $\Si$ be a topological sphere and $(I_i,II)$, $i=1,2$, two
Codazzi pairs on $\Sigma$ with the same extrinsic curvature $K>0$.
Then $I_1=I_2$.
\end{teorema}
\begin{proof}
Since $K>0$, we can assume (changing $II$ by $-II$ if necessary)
that $II$ is a Riemannian metric on $\Si$. Taking a local isothermal
parameter $z$ for $II$, we can write
\begin{eqnarray*}
&I_i=P_i\,dz^2+2\lambda_i\,|dz|^2+\overline{P_i}\,d\zb^2&\\
&II=2\rho\,|dz|^2,
\end{eqnarray*}
with $\rho>0$.

Hence, the mean and extrinsic curvatures of the pair $(I_i,II)$,
$i=1,2$ can be written as
\begin{equation}
H_i=\frac{\lambda_i\,\rho}{\lambda_i^2-|P_i|^2} \label{HK},\qquad
K=\frac{\rho^2}{\lambda_i^2-|P_i|^2},
\end{equation}
and the shape operator is given by
$$
S_i\frac{\partial\ }{\partial z}=\frac{K}{\rho}\,\left(\lambda_i\,\frac{\partial\
}{\partial z}-P_i\,\frac{\partial\ }{\partial \zb}\right).
$$
Let us denote by $\nabla^i$ the Levi-Civita connection associated to
the metric $I_i$ and put
$$
\nabla^i_{\frac{\partial\ }{\partial z}}\frac{\partial\ }{\partial
z}=\Gamma_{11}^{1,i}\,\frac{\partial\ }{\partial z}+
\Gamma_{11}^{2,i}\,\frac{\partial\ }{\partial \zb},\qquad \nabla^i_{\frac{\partial\
}{\partial z}}\frac{\partial\ }{\partial \zb}=\Gamma_{12}^{1,i}\,\frac{\partial\
}{\partial z}+ \overline{\Gamma_{12}^{1,i}}\,\frac{\partial\ }{\partial \zb}.
$$
Since $(I_i,II)$ is a Codazzi pair we have
\begin{eqnarray*}
0&=&I_i\left(\nabla^i_{\frac{\partial\ }{\partial
z}}S_i{\frac{\partial\ }{\partial \zb}},{\frac{\partial\ }{\partial
z}}\right)-I_i\left(\nabla^i_{\frac{\partial\ }{\partial
\zb}}S_i{\frac{\partial\ }{\partial z}},{\frac{\partial\ }{\partial z}}\right)\\
&=&\rho_z-I_i\left(S_i{\frac{\partial\ }{\partial
\zb}},\nabla^i_{\frac{\partial\ }{\partial z}}\frac{\partial\
}{\partial z}\right)+I_i\left(S_i{\frac{\partial\ }{\partial
z}},\nabla^i_{\frac{\partial\ }{\partial z}}\frac{\partial\
}{\partial \zb}\right)
\end{eqnarray*}
and so
\begin{equation}
\label{formula1}
\rho_z=\rho\left(\Gamma_{11}^{1,i}-\overline{\Gamma_{12}^{1,i}}\right).
\end{equation}

On the other hand, a direct calculation gives
$$
\left(\lambda_i^2-|P_i|^2\right)_z=2\,\left(\lambda_i^2-|P_i|^2\right)\,\left(\Gamma_{11}^{1,i}+
\overline{\Gamma_{12}^{1,i}}\right).
$$

With all of this, we obtain from (\ref{HK}) by differentiating $K$
\begin{equation}
\label{formula2}
\frac{K_z}{K}=2\frac{\rho_z}{\rho}-2\left(\Gamma_{11}^{1,i}+
\overline{\Gamma_{12}^{1,i}}\right).
\end{equation}
Therefore, from (\ref{formula1}) and (\ref{formula2}),
$$
\frac{K_z}{K}=-4\,\overline{\Gamma_{12}^{1,i}}=\frac{2K}{\rho^2}\left(P_i\,\overline{P_i}_z-\lambda_i\,{P_i}_{\zb}\right),
$$
or equivalently,
\begin{equation}\label{f1}
{P_i}_{\zb}=-\frac{1}{2K}\left(\lambda_i\,K_z+P_i\,K_{\zb}\right).
\end{equation}

Moreover, from (\ref{HK}) we get
\begin{eqnarray}
|\lambda_1-\lambda_2|&=&\frac{\rho}{K}\,|H_1-H_2|=\frac{\rho}{K}\,\sqrt{(H_1-H_2)^2} \label{f2} \\
&\leq& \frac{\rho}{K}\,\left|\sqrt{H_1^2-K}-\sqrt{H_2^2-K}\right|=
\left||P_1|-|P_2|\right| \leq |P_1-P_2|. \nonumber
\end{eqnarray}

Now, we can derive from (\ref{f1}) and (\ref{f2}) that
$$
|{P_1}_{\zb}-{P_2}_{\zb}|\leq\frac{|K_z|+|K_{\zb}|}{2K}\,|P_1-P_2|.
$$
Finally, using \cite[Main Lemma]{AdCT} or \cite[Lemma 2.7.1]{Jost}, we conclude that the
quadratic form $(P_1-P_2)dz^2$ vanishes identically on the
topological sphere $\Si$. Thus, $P_1\equiv P_2$ and, from (\ref{HK}), we get $\lambda_1=\lambda_2$. Or equivalently, $I_1=I_2$.
\end{proof}

\section{The Codazzi function}\label{s3}

Under certain natural conditions, it is possible to obtain important
consequences about a surface endowed with a fundamental pair
although the Codazzi equation is not satisfied. In order to study
these conditions, next we define the Codazzi tensor and the Codazzi
function, which will play an essential role in our study.

\begin{definicion}\label{d4}
Given a fundamental pair $(I,II)$ on a surface $\Si$ with associated
shape operator $S$, we will call Codazzi tensor of $(I,II)$ to the
map $T_S:\campo (\Si)\times \campo (\Si)\fl \campo (\Si)$ defined by
$$
T_S(X,Y)=\nabla_XSY-\nabla_YSX-S[X,Y],\qquad X,Y\in \campo (\Si).
$$
\end{definicion}

Although the definition above has been made in an abstract context,
the Codazzi tensor appears naturally in the study of isometric
immersions of surfaces. To be more precise, the Codazzi equation of
a surface isometrically immersed in a 3-dimensional manifold $M^3$
is
$$
\nabla_XSY-\nabla_YSX-S[X,Y]=-\overline{R}(X,Y)N,\qquad X,Y\in\campo
(\Si),
$$
where $N$ is the unit normal vector field of the immersion, $S$ the associated
shape operator and $\overline{R}$ the curvature tensor of $M^3$
$$
\overline{R}(X,Y)Z=\overline{\nabla}_X\overline{\nabla}_YZ-
\overline{\nabla}_Y\overline{\nabla}_XZ-\overline{\nabla}_{[X,Y]}Z,
$$
$\overline{\nabla}$ being the Levi-Civita connection of $M^3$.

A straightforward computation shows that the Codazzi tensor of a
fundamental pair on a surface satisfies the following properties:

\begin{lema}\label{l1}
Let $(I,II)$ be a fundamental pair on a surface $\Si$ with
associated shape operator $S$ and Codazzi tensor $T_S$. Then
\begin{enumerate}
\item $T_S$ is skew-symmetric, i.e. $T_S(X,Y)=-T_S(Y,X)$ for all $X,Y\in \campo
(\Si)$.
\item $T_S$ is ${\cal C}^{\infty}(\Sigma)$-bilineal, that is,
$$
T_S(f_1\,X_1+f_2\,X_2,Y)=f_1\,T_S(X_1,Y)+f_2\,T_S(X_2,Y)
$$
for all vector fields $X_1,X_2,Y\in \campo (\Si)$ and differentiable real functions $f_1,f_2$.
\item Moreover, given vector fields $X,Y\in \campo (\Si)$ and a differentiable real function $f$ on $\Si$
$$
T_{fS}(X,Y)=f\,T_S(X,Y)+X(f)\,SY-Y(f)\,SX.
$$
\end{enumerate}
\end{lema}

Associated to the Codazzi tensor of a fundamental pair we define the
Codazzi function, thanks to which we will {\it measure how distant}
the pair is from satisfying the Codazzi equation.

\begin{definicion}\label{d5}
Let $(I,II)$ be a fundamental pair on a surface $\Si$ with
associated shape operator $S$. We will call Codazzi function of
$(I,II)$ to the map ${\cal T}_S:\Si\fl\r$ given by
\begin{equation}\label{funcioncodazzi}
I\left(T_S(v_1,v_2),T_S(v_1,v_2)\right)={\cal T}_S(p)\,
\left(I(v_1,v_1)\,I(v_2,v_2)-I(v_1,v_2)^2\right),
\end{equation}
where $v_1,v_2\in T_p\Si$, $p\in\Si$.
\end{definicion}

Observe that ${\cal T}_S$ is a well-defined differentiable function
since the Codazzi tensor is skew-symmetric. Besides, ${\cal T}_S$
vanishes identically if, and only if, $(I,II)$ is a Codazzi pair.

\begin{lema}\label{anterior}
Let $(I,II)$ be a fundamental pair on a surface $\Si$ with
associated shape operator $S$, mean curvature $H$ and extrinsic
curvature $K$. Let $z$ be a local conformal parameter for $I$ such
that
$$
I=2\lambda\,|dz|^2,\qquad II=Q\,dz^2+2H\,\lambda\,|dz|^2+\overline{Q}\,d\zb^2.
$$
Then
$$
|Q_{\zb}|^2=\frac{\lambda\,{\cal T}_{\widetilde{S}}}{2(H^2-K)}\,|Q|^2,
$$
where $\widetilde{S}$ is the traceless operator $S-H\,Id$, $Id_p$
being the identity map on the tangent plane at $p\in \Si$.
\end{lema}
\begin{proof}
Since the Levi-Civita connection of $I$ is given by
(\ref{levicivita}) and the shape operator $S$ by
(\ref{endomorfismoweingarten}), we have
\begin{eqnarray}
 T_{\widetilde{S}}\left(\frac{\partial\ }{\partial z},\frac{\partial\ }{\partial
\zb}\right)&=&\nabla_{\frac{\partial\ }{\partial z}}\widetilde{S}\frac{\partial\
}{\partial \zb}-\nabla_{\frac{\partial\ }{\partial \zb}}\widetilde{S}\frac{\partial\
}{\partial z}\ =\ \nabla_{\frac{\partial\ }{\partial
z}}\frac{\overline{Q}}{\lambda}\frac{\partial\ }{\partial z}-\nabla_{\frac{\partial\
}{\partial \zb}}\frac{Q}{\lambda}\frac{\partial\ }{\partial \zb}\nonumber\\
&=&\frac{1}{\lambda}\left(\overline{Q}_z\,\frac{\partial\ }{\partial z}-
Q_{\zb}\,\frac{\partial\ }{\partial \zb}\right).\label{primero1}
\end{eqnarray}

Hence, using (\ref{funcioncodazzi}) one gets
$\frac{2}{\lambda}|Q_{\zb}|^2={\cal T}_{\widetilde{S}}\,\lambda^2$.
The proof finishes using (\ref{modulodeq}).
\end{proof}

Given a fundamental pair $(I,II)$ on a surface $\Si$, we will denote
by $\Si_U\subseteq\Si$ the set of umbilical points of the pair. Then
we have

\begin{teorema}\label{thglover}
Let $(I,II)$ be a fundamental pair on a surface $\Si$ with
associated shape operator $S$, mean curvature $H$ and extrinsic
curvature $K$. Let us suppose that every point
$p\in\partial\Sigma_U$ has a neighborhood $V_p$ such that
$$
\frac{{\cal T}_{\widetilde{S}}}{H^2-K}\quad \mbox{is bounded in}\,
V_p\cap(\Si-\Si_U),
$$
where $\widetilde{S}=S-H\, Id$. Then either the Hopf differential of
$(I,II)$ vanishes identically or its zeroes are isolated and of
negative index.

In particular, if $\Si$ is a topological sphere then the pair is
totally umbilical.
\end{teorema}
\begin{proof}
If $p\in\partial\Sigma_U$, then there exists an open neighborhood
$V_p$ and a constant $m_0$ such that $\frac{{\cal T}_{\widetilde{S}}}{H^2-K}\leq m_0$ in
$V_p\cap(\Si-\Si_U)$. Thus, using Lemma \ref{anterior}
\begin{equation}\label{noseyo}
|Q_{\zb}|^2\ \leq\ m_0\,\frac{\lambda}{2}\,|Q|^2
\end{equation}
in $V_p\cap(\Si-\Si_U)$. Besides, since this inequality is also
valid in the interior of $\Sigma_U$, we conclude that (\ref{noseyo})
holds in $V_p$.

Therefore, using again \cite{AdCT} or \cite{Jost}, we have that $p$ is an
isolated zero of negative index of the Hopf differential, as we
wanted to prove.

To end up, if $\Si$ is a topological sphere the result follows from
the Poincaré index Theorem as in Theorem \ref{elteorema}.
\end{proof}

Again, we point out that if $\Si$ is a topological torus under the
assumptions above, then the pair is either totally umbilical or umbilically free. Analogously, if $\Si$ is a closed topological
disk and $\partial\Si$ is a line of curvature of $(I,II)$, then the
pair is totally umbilical.

If $(I,II)$ is a Codazzi pair we have
$$
T_{\widetilde{S}}\left(\frac{\partial\ }{\partial z},\frac{\partial\ }{\partial
\zb}\right)=-\nabla_{\frac{\partial\ }{\partial z}}H\frac{\partial\
}{\partial \zb}+\nabla_{\frac{\partial\ }{\partial \zb}}H\frac{\partial\
}{\partial z}=H_{\zb}\frac{\partial\
}{\partial z}-H_{z}\frac{\partial\
}{\partial \zb}
$$
and, therefore, its Codazzi function is
$$
{\cal T}_{\widetilde{S}}\,=\,\frac{2}{\lambda}\,|H_{z}|^2\,=\,\|\nabla H\|^2,
$$
where $\|\nabla H\|$ stands for the modulus of the gradient of $H$
for the Riemannian metric $I$. Thus, Theorem \ref{thglover} can be
applied for Codazzi pairs whenever the quotient $\|\nabla
H\|^2/(H^2-K)$ is bounded.

However, this result can also be applied to fundamental pairs which
are not Codazzi pairs, as was implicitly made in \cite{EGR} in order
to classify the complete surfaces with constant extrinsic curvature
in the product spaces $\h^2\times\r$ and $\s^2\times\r$.

\section{Holomorphic quadratic differentials.}\label{s4}

In this section we will see that, under certain assumptions on a
Codazzi pair, it is possible to obtain a new Codazzi pair with vanishing
constant mean curvature which is geometrically related to the first
one. Thanks to this second Codazzi pair, we will show the existence
of a holomorphic quadratic differential which will provide important
information on the geometric behavior of the initial pair.

If $(u,v)$ are doubly orthogonal parameters for a fundamental pair
$(I,II)$, then we can write
$$
I=E\,du^2+G\,dv^2,\qquad II=k_1\,E\,du^2+k_2\,G\,dv^2.
$$
Hence, the Codazzi tensor acting on the vector fields
$\frac{\partial\ }{\partial u},\frac{\partial\ }{\partial v}$ can be
expressed as
\begin{eqnarray}
T_S\left(\frac{\partial\ }{\partial u},\frac{\partial\ }{\partial
v}\right)&=&\nabla_{\frac{\partial\ }{\partial u}}S\frac{\partial\ }{\partial
v}-\nabla_{\frac{\partial\ }{\partial v}}S\frac{\partial\ }{\partial
u}\,=\,\nabla_{\frac{\partial\ }{\partial u}}k_2\frac{\partial\ }{\partial
v}-\nabla_{\frac{\partial\ }{\partial v}}k_1\frac{\partial\ }{\partial u}\nonumber\\
&=&(k_2)_u\,\frac{\partial\ }{\partial v}-(k_1)_v\,\frac{\partial\ }{\partial
u}+(k_2-k_1)\left(\frac{E_v}{2E}\frac{\partial\ }{\partial
u}+\frac{G_u}{2G}\frac{\partial\ }{\partial v}\right)\label{Tuv}\\
&=&-\frac{1}{E}\left((k_1E)_v-H\,E_v\right)\,\frac{\partial\ }{\partial u}+
\frac{1}{G}\left((k_2G)_u-H\,G_u\right)\,\frac{\partial\ }{\partial v}\nonumber\,,
\end{eqnarray}
where $H$ is the mean curvature of $(I,II)$.

We observe that we can take doubly orthogonal parameters in a
neighborhood of every non umbilical point as well as in a
neighborhood of every point in the interior of
$\Si_U=\{\mbox{umbilical points of }(I,II)\}$. Thus, the set of
points where there exist doubly orthogonal parameters is dense in
$\Si$. Consequently, all the properties that we prove by using this
kind of parameters, will be extended to the whole surface by
continuity.

Throughout this section we will use the new quadratic form $II'$ associated to the fundamental pair $(I,II)$ given by
$$
II'=II-H\,I.
$$

\begin{lema}\label{lemita}
Let $(I,II)$ be a Codazzi pair on a surface $\Si$ with mean and
extrinsic curvatures $H$ and $K$ respectively. Let $\varphi$ be a
differentiable function on $\Si$ such that the function
$\sinh\varphi/\sqrt{H^2-K}$ can be differentiably extended to $\Si$.
Then
$$
\begin{array}
{l} {\displaystyle A=\cosh\varphi\,I+\frac{\sinh\varphi}{\sqrt{H^2-K}}\,II'}\\[6mm]
{\displaystyle B=\sqrt{H^2-K}\sinh\varphi\,I+\cosh\varphi\,II'},
\end{array}
$$
is a fundamental pair with mean curvature $H(A,B)=0$, extrinsic
curvature $K(A,B)=-(H^2-K)$ and such that its Codazzi tensor
$T_{\widetilde{S}}$ satisfies
\begin{equation}\label{TStilde}
T_{\widetilde{S}}(X,Y)=\omega(Y)\,X-\omega(X)\,Y,\qquad\omega=\frac{1}{2}(dH-\sqrt{H^2-K}d\varphi),
\end{equation}
for all $X,Y\in\campo(\Si)$.
\end{lema}
\begin{proof}
Let $(u,v)$ be doubly orthogonal parameters for the Codazzi pair
$(I,II)$ such that
$$
I=E\,du^2+G\,dv^2,\qquad II=k_1\,E\,du^2+k_2\,G\,dv^2,
$$
being $k_1\geq k_2$.

Then we can write $(A,B)$ as
$$
A=e^\varphi\,E\,du^2+e^{-\varphi}\,G\,dv^2,\qquad
B=\frac{k_1-k_2}{2}\,\left(e^\varphi\,E\,du^2-e^{-\varphi}\,G\,dv^2\right).
$$

Hence, $A$ is a Riemannian metric on $\Si$ and the mean and extrinsic curvatures of
the pair are given by $H(A,B)=0$ and $K(A,B)=-(H^2-K)$.

In addition, from (\ref{Tuv}) and taking into account that $(I,II)$
is a Codazzi pair, we get
\begin{eqnarray*}
T_{\widetilde{S}}\left(\frac{\partial\ }{\partial u},\frac{\partial\ }{\partial
v}\right)&=&-\frac{1}{e^\varphi\,E}\left(\frac{k_1-k_2}{2}\,\left(e^\varphi\,E\right)\right)_v
\frac{\partial\ }{\partial
u}-\frac{1}{e^{-\varphi}\,G}\left(\frac{k_1-k_2}{2}\,\left(e^{-\varphi}\,G\right)\right)_u
\frac{\partial\ }{\partial v}\\
&=&-\frac{1}{2}\left((k_1)_v-(k_2)_v+(k_1-k_2)\varphi_v-2(k_1)_v)\right)\frac{\partial\
}{\partial u}\\
&&-\frac{1}{2}\left((k_1)_u-(k_2)_u-(k_1-k_2)\varphi_u+2(k_2)_u)\right)\frac{\partial\
}{\partial v}\\
&=&\omega\left(\frac{\partial\ }{\partial v}\right)\frac{\partial\ }{\partial
u}-\omega\left(\frac{\partial\ }{\partial u}\right)\frac{\partial\ }{\partial v}.
\end{eqnarray*}
Finally, by linearity we obtain (\ref{TStilde}).
\end{proof}

Under the assumptions of Lemma \ref{lemita}, if we take a conformal parameter $z$ for $A$ and put
\begin{equation}\label{elnuevopar}
A=2\lambda\,|dz|^2,\qquad B=Q\,dz^2+\overline{Q}\,d\zb^2,
\end{equation}
then from (\ref{primero1}) and (\ref{TStilde}) we get
$$
Q_{\zb}=\lambda\,\omega\left(\frac{\partial\ }{\partial
z}\right)=\frac{\lambda}{2}\left(H_z-\sqrt{H^2-K}\,\varphi_z\right).
$$

Moreover, we obtain from (\ref{TStilde}) that the pair $(A,B)$ given by (\ref{elnuevopar}) is a Codazzi pair if and only if $dH-\sqrt{H^2-K}d\varphi=0$, or equivalently $Q_{\zb}=0$.

With all of this we have
\begin{corolario}\label{corolario3}
Let $(I,II)$ be a Codazzi pair on a surface $\Si$ with mean and
extrinsic curvatures $H$ and $K$ respectively. Let $\varphi$ be a
differentiable function on $\Si$ such that the function $\sinh\varphi/\sqrt{H^2-K}$ can be differentiably extended
to $\Si$. Then the fundamental pair
\begin{equation}\label{AB}
\begin{array}
{l} 
{\displaystyle A=\cosh\varphi\,I+\frac{\sinh\varphi}{\sqrt{H^2-K}}\,II'}\\[6mm]
{\displaystyle B=\sqrt{H^2-K}\sinh\varphi\,I+\cosh\varphi\,II'},
\end{array}
\end{equation}
has mean curvature $H(A,B)=0$ and extrinsic curvature $K(A,B)=-(H^2-K)$. In addition, the following conditions are equivalent:
\begin{itemize}
\item $(A,B)$ is a Codazzi pair,
\item the Hopf differential of
$(A,B)$ is holomorphic for the conformal structure induced by $A$,
\item $
dH-\sqrt{H^2-K}d\varphi=0.$
\end{itemize}
\end{corolario}

Next we see some situations where the corollary above can be used.
In order to do that and following the classical notation, we give
the following definition
\begin{definicion}\label{sw}
We say that a Codazzi pair $(I,II)$ is a special Weingarten pair if there
exists a differentiable function $f$ defined on an interval ${\cal
J}\subseteq [0,\infty)$ such that its mean curvature $H$ and extrinsic curvature $K$
satisfy
$$
H=f(H^2-K).
$$
\end{definicion}

Now, let us suppose that the mean and extrinsic curvatures of a
Codazzi pair $(I,II)$ satisfy a general Weingarten relationship $W(H,K)=0$.
Let us parametrize by taking $H=H(t)$, $K=K(t)$, for $t$ varying in a
certain interval. Then, if we look for a solution of the type
$\varphi=\varphi(t)$ for the previous Equation $dH-\sqrt{H^2-K}d\varphi=0$, we have
$$
\sqrt{H(t)^2-K(t)}\,\varphi'(t)=H'(t).
$$

Therefore, if there exists a primitive $\varphi(t)$ of the function
$$\frac{H'(t)}{\sqrt{H(t)^2-K(t)}}$$
with $\sinh\varphi(t)/\sqrt{H(t)^2-K(t)}$ well-defined even at the
umbilical points, then the Codazzi pair $(A,B)$, given as (\ref{AB}), will exist on the whole surface $\Sigma$.

In addition, if $(I,II)$ is a special Weingarten pair satisfying
$H=f(H^2-K)$, then we can parametrize in the way $H^2-K=t^2$ and
$H=f(t^2)$. This allows us to take
\begin{equation}\label{varfi}
\varphi(t)=\int_0^t2f'(s^2)\,ds
\end{equation}
whenever there exist umbilical points (i.e. $t=0$ has sense), or any
primitive of $2f'(t^2)$ otherwise.

Thus, for special Weingarten pairs, the metric $A$ defined
as in Corollary \ref{corolario3} is always well-defined, because so
is the function $\sinh\varphi(t)/t$.

The metric $A$ for special Weingarten surfaces in $\r^3$ and $\h^3$
was first defined by R.L. Bryant in \cite{Br}. In
that work, he also found a holomorphic quadratic form for the metric $A$ which
agrees with the Hopf differential of the pair $(A,B)$. Thanks to it, Bryant
provided an easy proof of the fact that
every topological sphere in $\r^3$ or $\h^3$ satisfying a special
Weingarten relationship must be totally umbilical.

The abstract formulation which we have adopted in this work, allows
us to extend this result to general special  Weingarten
surfaces.

Another remarkable fact is that, in our abstract context, we are
able to recover every special Weingarten pair as follows

\begin{corolario}\label{corolario4}
Let $\Si$ be a surface and $f$ a differentiable function defined on
an interval ${\cal J}\subseteq [0,\infty)$. Let us take a primitive
$\varphi(t)$ of $2f'(t^2)$ on that interval such that the function
$\sinh\varphi(t)/t$ is well-defined. Then every special  Weingarten
pair $(I,II)$ on $\Si$ satisfying $H=f(H^2-K)$ is given by
\begin{equation}
\label{III}
\begin{array}
{l}
{\displaystyle I=-\frac{\sinh\varphi(t)}{t}\,Q+\cosh\varphi(t)\,A-
\frac{\sinh\varphi(t)}{t}\,\overline{Q},}\\[5mm]
{\displaystyle
II-f(t^2)\,I=-\cosh\varphi(t)\,Q+t\,\sinh\varphi(t)\,A-\cosh\varphi(t)\,\overline{Q},
}
\end{array}
\end{equation}
where $A$ is a Riemannian metric on $\Si$ and $Q$ a holomorphic
2-form for $A$ such that the image of the function
$t:\Si\fl[0,\infty)$ defined as $2\,|Q|=t\,A$ is contained in ${\cal
J}$. In particular, $t^2=H^2-K$.
\end{corolario}
\begin{proof}
It suffices to observe that given a special  Weingarten pair
$(I,II)$, if we take $H^2-K=t^2$ (and therefore $H=f(t^2)$), we have
already proved that there exists $\varphi=\varphi(t)$, primitive of
$2f'(t^2)$, in the conditions of Corollary \ref{corolario3}. Thus,
there exists a pair $(A,B)$ made up of a Riemannian metric $A$ and a
quadratic form $B$ which can be written as $B=Q+\overline{Q}$, since
$H(A,B)=0$. Besides, the Hopf differential $Q$ of $(A,B)$ is a
holomorphic 2-form for the metric $A$ and
$$
t^2\,=\,H^2-K\,=\,-K(A,B)\,=\,4\,\frac{|Q|^2}{|A|^2}.
$$

Summing up, using (\ref{AB}) it is possible to recover $(I,II)$ from
$(A,B)$ as (\ref{III}). Finally, it is a straightforward computation
to check that any pair $(A,B)$ as above, gives a Codazzi pair
$(I,II)$ which is special Weingarten.
\end{proof}

\section{Applications in Space Forms}\label{s5}

In this section we focus our attention on surfaces in space forms. We will obtain
several results as a consequence of the abstract study developed
previously.

We start giving a geometrical argument in order to obtain height
bounds for a large amount of families of surfaces which satisfy a
maximum principle. The proof is based on some ideas used in \cite{EGR}.

\begin{definicion}\label{d7}
We say that a family ${\cal A}$ of oriented surfaces in $\r^3$
satisfies the Hopf maximum principle if the following properties are
satisfied:
\begin{enumerate}
\item ${\cal A}$ is invariant under isometries of $\r^3$. In other words, if $\Si\in{\cal A}$
and $\varphi$ is an isometry of $\r^3$, then $\varphi(\Si)\in{\cal
A}$.
\item If $\Si\in{\cal A}$ and $\widetilde{\Si}$ is another surface contained in $\Si$, then
$\widetilde{\Si}\in{\cal A}$.
\item  There is an embedded compact surface without boundary in ${\cal A}$.
\item Whichever two surfaces in ${\cal A}$ satisfy the maximum principle (interior and boundary).
\end{enumerate}
\end{definicion}

Note that a large amount of families of surfaces verify the Hopf
maximum principle. Classical examples of this fact are the family of
surfaces with constant mean curvature $H\neq 0$ and the family of
surfaces with positive constant extrinsic curvature $K$. And, more
generally, the family of special Weingarten surfaces in $\r^3$
satisfying a relation of the type $H=f(H^2-K)$, where $f$ is a
differentiable function defined on an interval ${\cal J}\subseteq
[0,\infty)$ with $0\in {\cal J}$, such that $f(0)\neq 0$ and
$4tf'(t)^2<1$ for all $t\in{\cal J}$ (see \cite{RoS}).

We also point out that if a family of surfaces ${\cal A}$ satisfy
the Hopf maximum principle, then there exists, up to isometries of
$\r^3$, a unique embedded compact surface $\Si$ without boundary in
${\cal A}$. Such surface is, necessarily, a totally umbilical
sphere.

To see this, it suffices to observe that the Alexandrov reflection
principle works for surfaces in ${\cal A}$. Thus, for every plane
$P\subseteq\r^3$ there exists a plane, parallel to $P$, which is a
symmetry plane of $\Si$. Therefore, $\Si$ is a round sphere.

In addition, there cannot be two totally umbilical spheres
$\Si_1,\Si_2$ in ${\cal A}$ which are non isometric. Otherwise, up
to isometries, we can suppose that one of them, let us say $\Si_1$,
is contained in the bounded region determined by $\Si_2$. If we move
$\Si_1$ until it meets first $\Si_2$ and at this contact point the
normal vectors to $\Si_1,\Si_2$ coincide, we can conclude that
$\Si_1=\Si_2$ from the maximum principle. If the normal vectors at
that point do not coincide, we keep on moving $\Si_1$ until it meets
$\Si_2$ at a last contact point, where necessarily the normal
vectors do coincide, which allows us, as before, to assert that
$\Si_1=\Si_2$.

Now, let us see that there exists a constant $c_{\cal A}$ such that
for all compact surface $\Si\in{\cal A}$ whose boundary is contained
in a plane $P$, the maximum distance from a point $p\in\Si$ to $P$
is bounded by $c_{\cal A}$. This bound only depends on the radius of
the unique totally umbilical sphere contained in the family ${\cal
A}$.

Although we will not provide optimal estimates here, only the
existence of such height estimates respect to planes will allow us
to get interesting consequences regarding several aspects of
embedded surfaces in $\r^3$ (see \cite{KKMS,KKS,M,RoS}).

We will start studying graphs $\Sigma$ with boundary contained in a
plane $P$ of $\r^3$. Up to an isometry, we can assume that $P$ is
the $xy-$plane, and so
$$
\Sigma = \set{ (x,y , u(x,y)) \in \r ^3 : \, \, (x,y)\in \Omega \subseteq \r^2 }.
$$

\begin{teorema}\label{teoremaco}
Let ${\cal A}$ be a family of surfaces in $\r^3$ satisfying the Hopf
maximum principle, and $\Sigma\in{\cal A}$ a compact graph on a
domain $\Omega$ in the $xy-$plane with $\partial\Sigma$ contained in
this plane. Then for all $p\in\Si$, the distance in $\r ^3$ from $p$
to the $xy-$plane is less or equal to $4R_{\cal A}$. Here, $R_{\cal
A}$ stands for the radius of the unique totally umbilical sphere in
the family ${\cal A}$.
\end{teorema}
\begin{proof}
Let $\Sigma\in{\cal A}$ be a graph on a domain $\Omega$ in the
$xy-$plane and $\Sigma _0$ the unique totally umbilical sphere of
${\cal A}$. Let $P(t)$ be the foliation of $\r^3$ by horizontal
planes, $P(t)$ being the plane at height $t$.

Let us see that for every $t > 2R_{\cal A}$, the diameter of any
open connected component bounded by $\Sigma (t) = P(t)\cap \Sigma $
is less than or equal to $2R_{\cal A}$.

Indeed, let us suppose that this assertion is not true. Then, for
some connected component $C(t)$ of $\Sigma (t)$, there are points
$p,q$ in the interior of the domain $\Omega (t)$ in $P(t)$ bounded
by $C(t)$ such that $\mbox{dist}(p,q)> 2R_{\cal A}$. Let $Q$ be the
domain in $\r^3$ bounded by $\Sigma \cup \Omega $. Let $\beta$ be a
curve in $\Omega (t)$ joining $p$ and $q$, $\beta$ and $C(t)$ being
disjoint. Let $\Pi$ be the \emph{rectangle} given by
$$
\Pi = \set{\alpha _s (r) \, : \, \, s \in \mathcal{I} , \, r \in [0, t]}
$$
where $\mathcal{I}$ is the interval where $\beta$ is defined, and
$\alpha _s$ is the geodesic with initial data $\alpha _s (0) = \beta
(s)$ and $\alpha _s ' (0) = -e_3$, $r$ being the length arc parameter
along $\alpha _{s}$ and $e_3=(0,0,1)$.

Since $\Sigma $ is a graph and $\beta $ is contained in the interior
of the domain determined by $C(t)$, then $\Pi\subset Q$. Let
$\widetilde{p}\in \Pi$ be a point whose distance to $\partial \Pi $
is greater than $R_{\cal A}$. Note that, according to our
construction of $\Pi$, the point $\widetilde{p}$ necessarily exists.

Let $\eta(r)$ be a horizontal geodesic passing through
$\widetilde{p}$ and such that every point in $\eta(r)$ is far from
$\partial\Pi$ a distance greater than $R_{\cal A}$. Observe that such a
geodesic can be chosen as the horizontal line in the plane $P(t_1)$
containing the point $\widetilde{p}$ and being orthogonal to the
vector joining $p$ and $q$. Let $\widetilde{q_1}$ be the first point
where $\eta$ meets $Q$, and $\widetilde{q_2}$ the last one.

Now, let us consider the spheres $\Si_0(r)\in{\cal A}$ centered at
every point $\eta(r)$. Note that these spheres can be obtained from
the rotational sphere $\Sigma _0$ by means of a translation of $\r
^3$.

There exists a first sphere in this family (coming from
$\widetilde{q_1}$) which meets $\Si$. If the normal vectors of both
surfaces coincide at this point, we conclude that both surfaces
agree by the maximum principle. On the other hand, if the normal
vectors are opposite, we reason as follows.

Let us consider the first sphere $\Si_0(r_0)$ in the family above
(coming from $\widetilde{q_1}$) which meets $\Pi$ at an interior point
of $\Pi$.

For every $r>r_0$ we consider the piece $\widetilde{\Si}_0(r)$ of
the sphere $\Si_0(r)$ which has gone through $\Pi$. Since these
spheres leave $Q$ at $\widetilde{q_2}$
and none of them meets $\partial\Pi$, there exists a first value
$r_1$ such that $\widetilde{\Si}_0(r_1)$ meets first $\partial Q
\cap \Sigma $ at a point $\widetilde{q_0}$. Thus, applying the
maximum principle to $\Si_0(r_1)$ and $\Si$ at $\widetilde{q}_0$, we conclude
that both surfaces agree, which is a contradiction.

Therefore we obtain that, for height  $t=2R_{\cal A}$, the diameter
of every open connected component bounded by $\Sigma (t) = P(t)\cap
\Sigma $ is less than or equal to $2R_{\cal A}$.\vspace{2mm}

To finish, we will see that $P(t)\cap \Sigma$ is empty for $t>
4R_{\cal A}$. To do that, it suffices to prove the following assertion
\begin{quote}
Let $\Omega_1$ be a connected component bounded by $\Sigma (2R_{\cal
A})$ in $P(2R_{\cal A})$. The distance from any point in $\Sigma$
(which is a graph on $\Omega_1$) to the plane $P(2R_{\cal A})$ is
less than or equal to the diameter of $\Omega_1$.
\end{quote}

Let $\sigma$ be a support line of $\partial\Omega_1$ in $P(2R_{\cal
A})$ with exterior unitary normal vector $v$, and let us take
$\eta(r)$ a geodesic such that $\eta(0)\in\sigma$ and
$\eta'(0)=\frac{1}{\sqrt{2}}(v+e_3)$.

Now, let us consider for every $r$ the plane $\Pi(r)$ in $\r^3$
passing through $\eta(r)$  which is orthogonal to
$\eta'(r)=\eta'(0)$. Such planes intersect every horizontal plane in
a line parallel to $\sigma$, being $\pi/2$ the angle between them.

If the assertion above was not true, there would exist a point
$p\in\Si$ over $\Omega_1$
such that its height on the plane $P(2R_{\cal A})$ would be greater
than the diameter of $\Omega_1$.

Let $\Si_1$ be the compact piece of $\Si$ which is a graph on $\Omega_1$.
Observe that, for $r$ big enough, $\Pi(r)$ does not meet $\Si_1$. In
addition, for $r=0$ the plane $\Pi(0)$ contains the line $\sigma$,
and the reflection of $p$ with respect to $\Pi(0)$
is a point whose vertical projection on $P(2R_{\cal A})$ is not in
$\Omega_1$. Therefore, using the Alexandrov reflection principle for
the planes $\Pi(r)$ with $r$ coming from infinity, there exists a
first value $r_0>0$ such that either the reflection of the piece of
$\Si_1$ which is over $\Pi(r)$ meets first $\Si_1$ at an interior
point or both surfaces are tangent at a point in the boundary. But
this is a contradiction, by the maximum principle.

This finishes the proof.
\end{proof}

As a consequence of this result, we are able to bound the maximum distance
attained by an embedded compact surface whose boundary is contained
in a plane.

\begin{corolario}\label{corolariaco}
Let ${\cal A}$ be a family of surfaces in $\r^3$ satisfying the Hopf
maximum principle. Then every embedded compact surface
$\Sigma\in{\cal A}$ whose boundary is contained in a plane $P$
verifies that for every $p\in\Si$ the distance in $\r ^3 $ from $p$
to the plane $P$ is less than or equal to $8R_{\cal A}$. Here,
$R_{\cal A}$ denotes the radius of the unique totally umbilical
sphere contained in ${\cal A}$.
\end{corolario}

This result follows from Theorem \ref{teoremaco} as a
standard consequence of the Alexandrov reflection principle for
planes parallel to $P$.

\begin{observacion}
The techniques used in Theorem \ref{teoremaco} and Corollary
\ref{corolariaco} are valid not only in $\r^3$, but also more
generally for hypersurfaces in $\r ^n$. Even more, they can easily
be adapted to study hypersurfaces in $\h ^n$.
\end{observacion}

The existence of a maximum principle and height estimates with respect to
planes for a family of surfaces ${\cal A}$, allow us to extend the
theory developed by Korevaar, Kusner, Meeks and Solomon
\cite{KKMS,KKS,M} for constant mean curvature surfaces in
$\r^3$ and $\h^3$ to our family ${\cal A}$. On the other hand, in
\cite{RoS} Rosenberg and Sa Earp showed that those techniques are
also suitable to study some families of surfaces satisfying a relationship of the type $H=f(H^2-K)$. However,
they do not use that the surfaces satisfy $H=f(H^2-K)$ actually,
but only that they satisfy the Hopf maximum principle and there
exist height estimates for them. Thus, following \cite{RoS} we get

\begin{teorema}[{\bf Cylindrical bounds}]\label{cilindrico}
Let ${\cal A}$ be a family of surfaces in $\r^3$ satisfying the Hopf
maximum principle. Let us take $\Sigma\in{\cal A}$ an annulus, i.e.
$\Si$ homeomorphic to a punctured closed disc of $\r^2$. If $\Si$ is
properly embedded, then it is contained in a half-cylinder of
$\r^3$.
\end{teorema}

A unitary vector $v\in\s^2$ is said to be an axial vector for
$\Si\subseteq\r^3$ if there exists a sequence of points $p_n\in\Si$
such that $|p_n|\rightarrow\infty$ and $p_n/|p_n|\rightarrow v$. In
particular, the theorem above asserts that for any properly embedded
annulus there exists a unique axial vector. In addition, this vector
is the generator of the rulings of the cylinder.

Finally, following \cite{RoS} for properly embedded complete
surfaces, we have

\begin{teorema}\label{finales}
Let ${\cal A}$ be a family of surfaces in $\r^3$ satisfying the Hopf
maximum principle. If $\Sigma\in{\cal A}$ is a properly embedded
surface with finite topology in $\r ^3$, then every end of $\Si$ is
cylindrically bounded. Moreover, if $a_1 , \ldots , a_k$ are the $k$
axial vectors corresponding to the ends, then these vectors cannot
be contained in an open hemisphere of $\s ^2$. In particular,
\begin{itemize}
\item $k =1$ is impossible.
\item If $k =2$, then $\Sigma$ is contained in a cylinder and is a rotational surface with respect to
a line parallel to the axis of the cylinder.
\item If $k=3$, then $\Sigma$ is contained in a slab.
\end{itemize}
\end{teorema}

\begin{definicion}\label{d8}
Let $(I,II)$ be a Codazzi pair on a surface $\Sigma$. We will say
that the pair is special Weingarten of elliptic type if its mean and
extrinsic curvatures $H$ and $K$ satisfy that $H=f(H^2-K)$, where
$f$ is a differentiable function defined on $[0,a)$,
$0<a\leq\infty$, such that $$4tf'(t)^2<1$$ for all $t\in[0,a)$.
\end{definicion}

It was proved by Rosenberg and Sa Earp \cite{RoS} (see also \cite{BSE}) that the set of
Weingarten surfaces of elliptic type in $\r^3$ and $\h^3$ with
$f(0)\neq 0$ is a family satisfying the Hopf maximum principle.
Thus, the above theorems are true for this kind of surfaces.
Actually these results were also proved in \cite{RoS}, although
under the additional hypothesis $f'(t)(1-2f(t)f'(t))\geq 0$.

The special Weingarten surfaces in $\r^3$ and $\h^3$ satisfying
$H=f(H^2-K)$ have been widely studied. In particular, an exhaustive
study of the rotational surfaces was developed by Sa Earp and
Toubiana \cite{ST5,ST2,ST3,ST4}.

In \cite{ST3} was posed the question of classifying the surfaces
satisfying $H=f(H^2-K)$ whose extrinsic curvature does not change
signs. More specifically, it is asked if such surfaces are totally
umbilical spheres, cylinders or surfaces of minimal type (i.e. with
$f(0)=0$). Observe that this fact is known for surfaces with
constant mean curvature. In fact, a minimal surface has non-positive
extrinsic curvature at every point and a complete surface with non
zero constant mean curvature and whose extrinsic curvature does not
change signs, must be a sphere or a cylinder \cite{Hof,KO}.

Next, and as a consequence of the study developed for Codazzi pairs,
we study that problem for the general case of special Weingarten
surfaces of elliptic type.

\begin{teorema}\label{otroth}
Let $\Sigma$ be a special Weingarten surface of elliptic type in
$\r^3$ satisfying that $H=f(H^2-K)$. Let us suppose that its
extrinsic curvature does not change signs:
\begin{enumerate}
\item If $\Sigma$ is complete and $K\geq 0$ at every point, then $\Sigma$ is a totally umbilical sphere,
a plane or a right circular cylinder.
\item If $\Sigma$ is properly embedded and $K\leq 0$ at every point, then $\Sigma$ is either a
right circular cylinder or a surface of minimal type (i.e.
$f(0)=0$).
\end{enumerate}
\end{teorema}

In order to prove this theorem, we will first establish the
following general Lemma for Codazzi pairs.

\begin{lema}\label{lematecnico}
Let $(I,II)$ be a special Weingarten pair of elliptic type on a
surface $\Sigma$, with mean and extrinsic curvatures $H$ and $K$
respectively. If $H^2-K\neq0$ on $\Sigma$, then the new metric
$$
g_0=\sqrt{H^2-K}\ A
$$
is a flat metric on $\Sigma$. Here, $A$ is the metric given by
(\ref{AB}) for the function $\varphi$ defined in (\ref{varfi}).

Moreover, if $I$ is complete and $H^2-K\geq c_0>0$ then the metric
$g_0$ is complete. In particular, $\Sigma$ with the conformal
structure given by $A$ (or by $g_0$) is the complex plane, the
once punctured complex plane or a torus.
\end{lema}
\begin{proof}
From Corollary \ref{corolario4} we get that $2|Q|=t\,A$, where
$t=\sqrt{H^2-K}$ and $Q$ is a holomorphic quadratic form for $A$. Thus,
since $H^2-K>0$, $g_0=2|Q|$ is a well-defined flat metric  on
$\Sigma$.

Let us see that $g_0$ is complete if $I$ is complete. In
such a case $g_0$ would be a complete flat metric and, so, the universal Riemannian covering of $\Sigma$ for the
metric $g_0$ would be the Euclidean plane. Hence,
$\Sigma$ would be conformally equivalent to the complex plane, to the once punctured
complex plane or to a torus.

Observe that, from (\ref{III}), we get that
\begin{equation}\label{ladesigualdad}
I\leq 2\cosh\varphi(t)\,A.
\end{equation}
On the other hand, since $(I,II)$ is a special Weingarten pair of elliptic type it
follows that $4\,s^2\,f'(s^2)^2<1$ and so, from (\ref{varfi}),
$$
s^2\,\varphi'(s)^2<1,\quad {\rm or\ equivalently}\quad
-\frac{1}{s}<\varphi'(s)<\frac{1}{s}.
$$
Hence, by integrating between a fixed point $s_0>0$ and $s$ one gets
that there exists a constant $c_1>0$ such that $|\varphi(s)|\leq
|\log s|+c_1$. Therefore, since
$$
\lim_{s\rightarrow\infty}\frac{\cosh\log(s)}{s}=\frac{1}{2}
$$
and $t\geq\sqrt{c_0}$, we deduce the existence of a constant $c_2>0$
such that $\cosh\varphi(t)\leq c_2\,t$.

Finally, from (\ref{ladesigualdad}) it follows that
$$
I\,\leq \,2c_2\,t\,A\,=\,2c_2\,g_0,
$$
that is, $g_0$ is complete.
\end{proof}
\noindent {\it Proof of Theorem \ref{otroth}:} Firstly, let us
suppose that $\Sigma$ is a complete surface in $\r^3$ with $K\geq0$.

If $K$ vanishes identically, then it is easy to conclude that
$\Sigma$ is either a plane or a right circular cylinder (see
\cite{ST5}).

If there exists a point where the extrinsic curvature is positive,
then either $\Sigma$ is homeomorphic to a sphere or it is properly
embedded and homeomorphic to the plane \cite{Wu}. In addition, we have $f(0)\neq 0$ (see
\cite{ST5}).

In the first case
$\Sigma$ must be a totally umbilical sphere from Theorem
\ref{elteorema}. The second case is not possible from Theorem
\ref{finales} applied to our family of special Weingarten surfaces
with $H=f(H^2-K)$.

Now, let us suppose that $\Sigma$ is properly embedded and $K\leq0$.
Then we have that
$$
0\geq K=H^2-(H^2-K)=f(H^2-K)^2-(H^2-K).
$$
Hence, if $f(0)\neq0$, since the function $f(s)^2-s$ is continuous
for $s\geq0$ and takes a positive value at $s=0$, then there exists
$s_0>0$ such that $f(s)^2-s>0$ for $s\in[0,s_0]$. Consequently
$H^2-K\geq s_0>0$ on $\Sigma$ since $K\leq 0$.

From Lemma \ref{lematecnico}, $\Sigma$ is homeomorphic to the
plane, to the once punctured plane or to a torus. Using once again
Theorem \ref{finales}, $\Sigma$ cannot be homeomorphic to a plane.
In addition, every compact surface in $\r^3$ must have a point with
positive extrinsic curvature, and so $\Sigma$ cannot be homeomorphic
to a torus. With all of this, $\Sigma$ must be homeomorphic to the
once punctured plane and so, from Theorem \ref{finales}, it must
be a rotational surface and must be contained in a cylinder $C$ of
$\r^3$.

To finish, let us see that $\Sigma$ is a right circular cylinder. In
fact, up to an isometry of $\r^3$, we can suppose that $\Sigma$ is a
rotational surface with respect to the $z$-axis. Let us denote by
$$\alpha=\Sigma\cap\{(x,y,z)\in\r^3:\ x>0,\ y=0\}$$ a generatrix curve of
$\Sigma$. It is clear that $\alpha$ is a line of curvature of
$\Sigma$ and its signed-curvature on the plane $y=0$ changes signs
if and only if $K$ changes signs.

Since $K\leq 0$, the sign of the
curvature of $\alpha$ on the plane $y=0$ does not change, and so
$\alpha$ is a convex curve. But, since $\alpha$ is contained in the
strip determined by the $z$-axis and the line parallel to
$C\cap\{(x,y,z)\in\r^3:\ x>0,\ y=0\}$, we conclude that $\alpha$
must be a line parallel to the $z$-axis, as we wanted to prove.
\hfill{\large $\Box$}

\footnotesize The first author is partially supported by Junta de Comunidades de
Castilla-La Mancha, Grant No PCI-08-0023. The second and third authors are partially
supported by Grupo de Excelencia P06-FQM-01642 Junta de Andalucía. The authors are
partially supported by MCYT-FEDER, Grant No MTM2007-65249
\end{document}